\documentclass[12pt,a4paper,leqno]{amsart} %optio titlepage!
\usepackage[T1]{fontenc}       % kirjaimet, joissa aksentteja (skandit)
\usepackage{amsfonts}          % AMS-fontit
\usepackage{amsmath}
\usepackage{amssymb}
\usepackage{overpic}
\usepackage{euscript} % \mathcal tuottaa hyvдnnдkцisen fontin
\usepackage{eurosym}
\usepackage{comment}
\usepackage{mathrsfs} % calligraphic font via the command \mathscr{ABC}
\usepackage{ifthen}
\usepackage{esint} %k.a.integraali komennolla \fint
\usepackage{color}

\long\def\symbolfootnote[#1]#2{\begingroup%
\def\thefootnote{\fnsymbol{footnote}}\footnote[#1]{#2}\endgroup}

{\qed\vspace{5pt}}

\usepackage{amsthm}

\newtheoremstyle{lause}% name
{5pt}% space above
{5pt}% space below
{\slshape}% body font
{\parindent}% indent amount (empty = no indent)
{\bfseries}% theorem head font
{.}% punctuation after theorem head
{.5em}% space after theorem head
{}% theorem head spec (can be left empty, meaning 'normal')

\theoremstyle{lause}
\newtheoremstyle{maaritelma}% name
{5pt}% space above
{5pt}% space below
{\rmfamily}% body font
{\parindent}% indent amount (empty = no indent)
{\bfseries}% theorem head font
{.}% punctuation after theorem head
{.5em}% space after theorem head
{}% theorem head spec (can be left empty, meaning 'normal')

\theoremstyle{maaritelma}
\newtheoremstyle{lause}% name
{5pt}% space above
{5pt}% space below
{\slshape}% body font
{\parindent}% indent amount (empty = no indent)
{\bfseries}% theorem head font
{.}% punctuation after theorem head
{.5em}% space after theorem head
{}% theorem head spec (can be left empty, meaning 'normal')

\theoremstyle{lause}
\newtheorem{theorem}{Theorem}[section]
\newtheorem{lemma}[theorem]{Lemma}
\newtheorem{proposition}[theorem]{Proposition}
\newtheorem{corollary}[theorem]{Corollary}
\newtheorem{problem}[theorem]{Problem}

\newtheoremstyle{maaritelma}% name
{5pt}% space above
{5pt}% space below
{\rmfamily}% body font
{\parindent}% indent amount (empty = no indent)
{\bfseries}% theorem head font
{.}% punctuation after theorem head
{.5em}% space after theorem head
{}% theorem head spec (can be left empty, meaning 'normal')

\theoremstyle{maaritelma}
\newtheorem{definition}[theorem]{Definition}

\newtheorem{example}[theorem]{Example}
\newtheorem{remark}[theorem]{Remark}

\numberwithin{equation}{section}

\begin{document}

\thispagestyle{empty}

\begin{center}

{\large{\textbf{Balayage of measures on a~locally compact space}}}

\vspace{18pt}

\textbf{Natalia Zorii}

\vspace{18pt}

\emph{Dedicated to Professor Bent Fuglede on the occasion of his 95th birthday}\vspace{8pt}

\footnotesize{\address{Institute of Mathematics, Academy of Sciences
of Ukraine, Tereshchenkivska~3, 01601,
Kyiv-4, Ukraine\\
natalia.zorii@gmail.com }}

\end{center}

\vspace{12pt}

{\footnotesize{\textbf{Abstract.} We develop a theory of inner balayage of a positive Radon measure $\mu$ of finite energy on a locally compact space $X$ to arbitrary $A\subset X$, generalizing Cartan's theory of Newtonian inner balayage on $\mathbb R^n$, $n\geqslant3$, to a suitable function kernel on $X$. As an application of the theory thereby established, we show that if the space $X$ is perfectly normal and of class $K_\sigma$, then a recent result by B.~Fuglede (Anal.\ Math., 2016) on outer balayage of $\mu$ to quasiclosed $A$ remains valid
for arbitrary Borel $A$. We give in particular various alternative definitions of inner (outer) balayage, provide a formula for evaluation of its total mass, and prove convergence theorems for inner (outer) swept measures and their potentials. The results obtained do hold (and are new in part)
for most classical kernels on $\mathbb R^n$, $n\geqslant2$, which is important in applications.}}
\symbolfootnote[0]{\quad 2010 Mathematics Subject Classification: Primary 31C15.}
\symbolfootnote[0]{\quad Key words: Radon measures on a locally compact space, inner and outer balayage, consistent kernel, energy principle, first and second maximum principles.}

\vspace{6pt}

\markboth{\emph{Natalia Zorii}} {\emph{Balayage of measures on a locally compact space}}

\section{Introduction}\label{sec1}

The paper deals with balayage of a positive Radon measure $\mu$ of finite energy on a locally compact (Hausdorff) space $X$ to a set $A\subset X$ in the setting of potentials with respect to a symmetric, lower semicontinuous (l.s.c.)\ kernel $\kappa:X\times X\to[0,\infty]$. Throughout the present section as well as Sects.~\ref{sec-inner}--\ref{sec-outer},
the kernel $\kappa$ is assumed to satisfy the energy, consistency, and domination principles \cite{F1,O}.

It has recently been shown by Bent Fuglede \cite[Theorem~4.12]{Fu5} that if a set $A$ is {\it quasiclosed\/} (that is, if it can be approximated in outer capacity by closed sets), then there is a unique positive Radon measure $\mu^{*A}$ of finite energy that is concentrated on $A$ (that is, $\mu^{*A}\in\mathcal E^+_A$) and has the property
\[\kappa\mu^{*A}=\kappa\mu\text{ \ q.e.\ on $A$},\] where $\kappa\nu(\cdot):=\int\kappa(\cdot,y)\,d\nu(y)$ is the {\it potential\/} of a Radon measure $\nu$ on $X$, and {\it q.e.}\ ({\it qua\-si-ev\-ery\-whe\-re\/}) means that the equality holds everywhere on $A$ except for a subset of outer capacity zero. Such a $\mu^{*A}$ is said to be the {\it outer balayage\/} of $\mu$ onto $A$.

We shall show below that if the space $X$ is perfectly normal and of class $K_\sigma$, then the quoted Fuglede's result remains valid for arbitrary {\it Borel\/} $A$ (Sect.~\ref{sec-outer}).
The outer balayage $\mu^{*A}$ now, however, is no longer concentrated on the set $A$ itself (as it was for $A$ quasiclosed), but on the closure of $A$ in $X$.
It is still characterized uniquely by the above display, but now within $\mathcal E'_A$, the closure of $\mathcal E^+_A$ in the topology determined by the {\it energy norm\/} $\|\nu\|:=\sqrt{\int\kappa\nu\,d\nu}$. The outer balayage $\mu^{*A}$ is actually the (unique) limit of the net $(\mu^{*K})$ in both the vague and the energy norm topologies when $K$ increases along the upper directed family of all compact subsets of $A$. Furthermore, it can alternatively be determined as a solution (which exists and is unique)
to either of the following two extremal problems.

\begin{problem}\label{pr-outer} In the class\/ $\Lambda^*_{A,\mu}$ of positive Radon measures\/ $\nu$ on\/ $X$ of finite energy and such that\/ $\kappa\nu\geqslant\kappa\mu$ q.e.\ on\/ $A$, find\/ $\lambda_{A,\mu}^*\in\Lambda^*_{A,\mu}$ of minimal potential:
\[\kappa\lambda_{A,\mu}^*=\min_{\nu\in\Lambda_{A,\mu}^*}\,\kappa\nu\text{ \ everywhere on $X$}.\]
\end{problem}

\begin{problem}\label{pr-proj}Find\/ $\mu_{\mathcal E'_A}\in\mathcal E'_A$ such that
\[\|\mu-\mu_{\mathcal E'_A}\|=\min_{\nu\in\mathcal E'_A}\,\|\mu-\nu\|=\inf_{\nu\in\mathcal E^+_A}\,\|\mu-\nu\|.\]
\end{problem}

If moreover $\mu$ is {\it bounded\/} (i.e.\ $\mu(X)<\infty$) while $\kappa$ satisfies Frostman's maximum principle,
then the outer balayage $\mu^{*A}$ can also be found as the (unique) solution to the problem obtained from Problem~\ref{pr-outer}, resp.\ Problem~\ref{pr-proj}, by requiring additionally that all the $\nu$ involved have the total mass
$\nu(X)\leqslant q$, $\mu(X)\leqslant q<\infty$ being given.

These results are obtained by a direct application of the theory of {\it inner balayage}, developed in Sects.~\ref{sec-inner}--\ref{proofs2} below and generalizing H.~Cartan's theory \cite{Ca2} of inner Newtonian balayage on $\mathbb R^n$, $n\geqslant3$, to a locally compact (l.c.)\ space $X$ endowed with a kernel described above. (In this part of the study, both $X$ and $A$ are {\it arbitrary}.)

We give in particular various alternative definitions of inner (outer) balayage, provide a formula for evaluation of its total mass, and prove convergence theorems for inner (outer) swept measures and their potentials (see Sects.~\ref{sec-inner}, \ref{sec-further}, \ref{sec-outer}).

The results obtained do hold (and are new in part) for the Newtonian kernel $|x-y|^{2-n}$ and, more generally, the $\alpha$-Riesz kernel $|x-y|^{\alpha-n}$ of order $0<\alpha\leqslant2$ on $\mathbb R^n$, $n\geqslant3$, as well as for the associated $\alpha$-Green kernel $g_D^\alpha$ on an open set $D\subset\mathbb R^n$.\footnote{For the theory of outer, resp.\ inner, Riesz balayage, we refer the reader to \cite{BH}, resp.\ \cite{Z-bal,Z-bal2}, the investigation in \cite{BH} being carried out in the general framework of balayage spaces.} The same is true for the ($2$-)Green kernel on a planar Greenian set.\footnote{Regarding the logarithmic kernel on $\mathbb R^2$, see Example~\ref{rem:clas} below and footnote~\ref{f-obstacle} attached to it.} This suggests that the present work can be useful, for instance, in connection with  optimal point configurations and minimum energy problems (see e.g.\ \cite{Z4,Dr,Z9,ZPot3,Br,DFHSZ2,FZ-Pot2,DFHSZ1,Z-AMP} for some applications of balayage, the papers \cite{Z4}--\cite{ZPot3} dealing with minimum energy problems with respect to a general function kernel $\kappa$ on a l.c.\ space $X$).

To begin with, we review some basic facts of the theory of potentials on a l.c.\ space $X$, using the fundamental study by Fuglede \cite{F1} as a guide.

\section{Basic facts of potential theory on a locally compact space}\label{sec11}Denote by $\mathfrak M=\mathfrak M(X)$ the linear space of all (real-val\-ued scalar Radon) measures $\mu$ on  a l.c.\ space $X$, equipped with the {\it vague\/} ($=${\it weak$^*$}) topology of pointwise convergence on the class $C_0=C_0(X)$ of all continuous functions $f:X\to\mathbb R$ with compact support, and by $\mathfrak M^+=\mathfrak M^+(X)$ the cone of all positive $\mu\in\mathfrak M$.

\begin{lemma}[{\rm see e.g.\ \cite[Section~1.1]{F1}}]\label{lemma-semi}For any l.s.c.\ function\/ $\psi:X\to[0,\infty]$, the mapping\/ $\mu\mapsto\int\psi\,d\mu$ is
vaguely l.s.c.\ on\/ $\mathfrak M^+$\/ {\rm(the integral here being understood as upper integral, see e.g.\ \cite[Section~IV.1.1]{B2})}.\end{lemma}

Given a (positive, symmetric, l.s.c.)\ kernel $\kappa$ on $X$ and given (signed) measures $\mu,\nu\in\mathfrak M$, define the {\it potential\/} and the {\it mutual energy\/}  by
\begin{align*}\kappa\mu(x)&:=\int\kappa(x,y)\,d\mu(y),\quad x\in X,\\
\kappa(\mu,\nu)&:=\int\kappa(x,y)\,d(\mu\otimes\nu)(x,y),
\end{align*}
respectively, provided the right-hand side is well defined as a finite number or $\pm\infty$ (for more details, see e.g.\ \cite[Section~2.1]{F1}). For $\mu=\nu$,
$\kappa(\mu,\nu)$ defines the {\it energy\/} $\kappa(\mu,\mu)$ of $\mu$.
In particular, if the measures are positive, then $\kappa\mu(x)$, resp.\ $\kappa(\mu,\nu)$, is well defined and represents a positive l.s.c.\ function of $(x,\mu)\in X\times\mathfrak M^+$, resp.\ of $(\mu,\nu)\in\mathfrak M^+\times\mathfrak M^+$ (the {\it principle of descent\/} \cite[Lemma~2.2.1]{F1}, cf.\ Lemma~\ref{lemma-semi}).

{\it Throughout the remainder of the present paper, we shall tacitly assume a kernel in question to satisfy the energy and consistency principles.} Recall that a kernel $\kappa$ is said to satisfy the {\it energy principle\/} (or to be {\it strictly positive definite\/}) if for any (signed) $\mu\in\mathfrak M$, $\kappa(\mu,\mu)$ is ${}\geqslant0$ whenever defined, and it is zero only for $\mu=0$.
All the $\mu\in\mathfrak M$ with finite energy\footnote{If the energy principle holds, then a (signed) measure $\mu\in\mathfrak M$ has finite energy if and only if so do both $\mu^+,\mu^-\in\mathfrak M^+$, the positive and negative parts of $\mu$ in the Hahn--Jor\-dan decomposition.} then form a pre-Hil\-bert space $\mathcal E=\mathcal E(X)$ with the inner product $(\mu,\nu):=\kappa(\mu,\nu)$ and the norm $\|\mu\|:=\sqrt{\kappa(\mu,\mu)}$, see \cite[Section~3.1]{F1}. The (Hausdorff) topology on $\mathcal E$ defined by the norm $\|\cdot\|$ is said to be {\it strong}.

We shall use the following lemma from the geometry of the pre-Hilbert space $\mathcal E$.

\begin{lemma}[{\rm see \cite[Lemma~4.1.1]{F1}}]\label{4.1.1}Suppose that\/ $\Gamma$ is a convex subset of\/ $\mathcal E$, and there exists\/ $\lambda\in\Gamma$ of minimal norm:
\[\|\lambda\|=\min_{\nu\in\Gamma}\,\|\nu\|.\]
For any\/ $\nu\in\Gamma$, then
\[\|\nu-\lambda\|^2\leqslant\|\nu\|^2-\|\lambda\|^2.\]
\end{lemma}

A (strictly positive definite) kernel $\kappa$ is called {\it consistent\/} \cite[Section~3.3]{F1} if every strong Cauchy net in the cone $\mathcal E^+=\mathcal E^+(X):=\mathcal E\cap\mathfrak M^+$ converges strongly to any of its vague limit points.\footnote{As the vague topology on $\mathfrak M^+$ does not satisfy the first axiom of
countability, the vague convergence cannot in general be described in terms of sequences. We follow Moore and Smith's theory of convergence
\cite{MS}, based on the concept of {\it nets}. However, if the space $X$ has a countable base of open sets, the use of nets in $\mathfrak M^+$ may be avoided. Indeed, such an $X$ is metrizable (with a metric $\varrho_X$) and of class $K_\sigma$ \cite[Section~IX.2, Corollary to Proposition~16]{B3}, and hence it has a countable dense subset $(x_k)\subset X$ \cite[Section~IX.2, Proposition~12]{B3}. Therefore,
\[\mathcal U_k=\Bigl\{\nu\in\mathfrak M^+: \ \int\bigl(1-k\varrho_X(x_k,x)\bigr)^+\,d|\nu-\nu_0|(x)<
1/k\Bigr\},\quad k\in\mathbb N,\]
form a countable base of vague neighborhoods of $\nu_0\in\mathfrak M^+$. Here $|\mu|:=\mu^++\mu^-$, $\mu\in\mathfrak M$.\label{foot}} The cone $\mathcal E^+$ then becomes {\it strongly complete}, a strongly bounded part of $\mathcal E^+$ being vaguely compact by \cite[Lemma~2.5.1]{F1}. As the strong limit of a strong Cauchy net $(\mu_k)\subset\mathcal E^+$ is unique, all the vague limit points of such a $(\mu_k)$ must be equal. The vague topology on $\mathfrak M$ being Hausdorff, we thus conclude by applying \cite[Section~I.9.1, Corollary]{B1} that any strong Cauchy net in $\mathcal E^+$ converges to some (unique) limit both strongly and vaguely. In Fuglede's terminology \cite[Section~3.3]{F1}, a strictly positive definite, consistent kernel is, therefore, {\it perfect}.

Given $A\subset X$, denote by $\mathfrak M^+_A$ the class of all $\mu\in\mathfrak M^+$ {\it concentrated on\/}
$A$, which means that $A^c:=X\setminus A$ is locally $\mu$-neg\-lig\-ible, or equivalently, that $A$ is $\mu$-meas\-ur\-able and $\mu=\mu|_A$, $\mu|_A:=1_A\cdot\mu$ being the {\it trace\/} of $\mu$ to $A$ \cite[Section~IV.14.7]{E2}. (Note that for $\mu\in\mathfrak M^+_A$, the indicator function $1_A$ of $A$ is locally $\mu$-int\-egr\-able.) The total mass of $\mu\in\mathfrak M^+_A$ is $\mu(X)=\mu_*(A)$, $\mu_*(A)$ and $\mu^*(A)$ denoting the {\it inner\/} and {\it outer\/} measure of $A$, respectively. If moreover $A$ is closed, or if $A^c$ is contained in a countable union of sets $Q_k$ with $\mu^*(Q_k)<\infty$,\footnote{If the latter holds, $A^c$ is said to be $\mu$-$\sigma$-{\it finite\/} \cite[Section~IV.7.3]{E2}. This occurs e.g.\ if the measure $\mu$ is bounded, or if the space $X$ is of class $K_\sigma$.} then for any $\mu\in\mathfrak M^+_A$, $A^c$ is $\mu$-neg\-lig\-ible, i.e.\ $\mu^*(A^c)=0$. In particular, if $A$ is closed, then $\mathfrak M^+_A$ consists of all $\mu\in\mathfrak M^+$ with $S(\mu)\subset A$, $S(\cdot)$ being the {\it support\/} of a measure.

Denoting now $\mathcal E^+_A:=\mathfrak M^+_A\cap\mathcal E$, we define the {\it inner capacity\/} of $A$ by
\[c_*(A):=\Bigl[\inf_{\mu\in\mathcal E^+_A: \ \mu(X)=1}\,\kappa(\mu,\mu)\Bigr]^{-1}.\]
(The infimum over the empty set is interpreted as $+\infty$.)
Then \cite[p.~153, Eq.~(2)]{F1}
\begin{equation}\label{153}c_*(A)=\sup\,c_*(K)\text{ \ ($K\subset A$ compact)}.\end{equation}
Also, by homogeneity reasons (cf.\ \cite[Lemma~2.3.1]{F1}),
\begin{equation}c_*(A)=0\iff\mathcal E^+_A=\{0\}\iff\mathcal E^+_K=\{0\}\text{ \ for every compact $K\subset A$}.\label{2.3.1}\end{equation}

We are thus led to the following conclusion, to be often used in what follows.

\begin{lemma}\label{l-negl}Given\/ $\mu\in\mathcal E^+$ and a\/ $\mu$-measurable set\/ $A\subset X$, $A$ is locally\/ $\mu$-neg\-lig\-ible if\/ $c_*(A)=0$, and it is\/ $\mu$-neg\-lig\-ible if moreover it is\/ $\mu$-$\sigma$-fi\-ni\-te.\end{lemma}

Due to the perfectness of the kernel, for any $A$ with $c_*(A)<\infty$ (thus in particular for any compact $A$), there is a unique $\gamma_A\in\mathcal E^+$, called the {\it inner equilibrium measure\/} for $A$,
that minimizes the energy $\kappa(\nu,\nu)$ among all
$\nu\in\mathcal E^+$ having the property $\kappa\nu\geqslant1$ n.e.\ on $A$, where {\it n.e.}\ ({\it nearly everywhere\/}) means that the equality holds everywhere on $A$ except for a subset of inner capacity zero. Furthermore \cite[Theorem~4.1]{F1},
\begin{align*}&\kappa(\gamma_A,\gamma_A)=\gamma_A(X)=c_*(A),\\
&\kappa\gamma_A\geqslant1\text{ \ n.e.\ on $A$},\\
&\kappa\gamma_A\leqslant1\text{ \ on $S(\gamma_A)$}.
\end{align*}
Thus, if moreover Frostman's maximum principle holds,\footnote{A kernel $\kappa$ is said to satisfy {\it Frostman's maximum principle\/} (${}={}$the {\it first maximum principle\/}) if for any $\mu\in\mathfrak M^+$ with $\kappa\mu\leqslant1$ on $S(\mu)$, the same inequality holds on all of $X$. For the purposes of the present paper, it is enough, in fact, to suppose that this is true for any $\mu$ of {\it finite\/} energy.\label{F}}
then also
\begin{equation}\label{eqq}\kappa\gamma_A=1\text{ \ n.e.\ on $A$}.\end{equation}

Defining further the {\it outer capacity\/} of $A$ by
\[c^*(A):=\inf\,c_*(D),\]
where $D$ ranges over all open sets that contain $A$, we have $c_*(A)\leqslant c^*(A)$. A set $A$ is said to be {\it capacitable\/} if $c_*(A)=c^*(A)$, and then we shall simply write \[c(A):=c_*(A)=c^*(A);\]
this occurs e.g.\ whenever $A$ is open or compact. If $A$ is capacitable and $c(A)<\infty$, then the inner equilibrium measure $\gamma_A$ serves simultaneously as the (unique) {\it outer equilibrium measure\/} $\gamma_A^*$, minimizing the energy $\kappa(\nu,\nu)$ among all
$\nu\in\mathcal E^+$ with the property $\kappa\nu\geqslant1$ q.e.\ on $A$
(see \cite[Theorem~4.3, Lemma~4.3.4]{F1}).

Throughout Sects.~\ref{sec-inner}--\ref{sec-outer}, in addition to the (permanent) requirement of perfectness for the kernel $\kappa$, we shall tacitly assume that $\kappa$ satisfies the {\it domination principle\/} ($={}$the {\it second maximum principle\/}), which means that for any $\mu,\nu\in\mathcal E^+$ such that $\kappa\mu\leqslant\kappa\nu$ $\mu$-a.e., the same inequality holds on all of $X$.

\begin{example}\label{rem:clas} The Newtonian kernel $|x-y|^{2-n}$ and, more generally, the $\alpha$-Riesz kernel $|x-y|^{\alpha-n}$ of order $\alpha\in(0,2]$ on $\mathbb R^n$, $n\geqslant3$, satisfy the energy and consistency principles as well as the first and second maximum principles \cite[Theorems~1.10, 1.15, 1.18, 1.27, 1.29]{L}. So does the associated $\alpha$-Green kernel on an arbitrary open subset of $\mathbb R^n$, $n\geqslant3$ \cite[Theorems~4.6, 4.9, 4.11]{FZ}. The ($2$-)Green kernel on a planar Greenian set is likewise strictly positive definite \cite[Section~I.XIII.7]{Doob} and consistent \cite{E}, and it satisfies the first and second maximum principles (see \cite[Theorem~5.1.11]{AG}, \cite[Section~I.V.10]{Doob}). Finally, the logarithmic kernel $-\log\,|x-y|$ on a closed disc in $\mathbb R^2$ of radius ${}<1$ is strictly positive definite and satisfies Frostman's maximum principle \cite[Theorems~1.6, 1.16]{L}, and hence it is perfect by \cite[Theorem~3.4.2]{F1}. However, the domination principle then fails in general; it does hold only in a weaker sense where the measures $\mu,\nu$ involved in the ab\-ove-quo\-ted definition satisfy the additional requirement that $\nu(\mathbb R^2)\leqslant\mu(\mathbb R^2)$, cf.\ \cite[Theorem~3.2]{ST}.\footnote{Because of this obstacle, the theory of inner and outer balayage on a l.c.\ space $X$, developed in the present paper, does not cover the case of the logarithmic kernel on $\mathbb R^2$.\label{f-obstacle}}\end{example}

\section{Preliminaries}\label{sec:pr}

In all that follows, $\kappa$ denotes a positive, symmetric, l.s.c.,\ perfect kernel on a l.c.\ space $X$.\footnote{For the notation and terminology used here, see Sects.~\ref{sec1}, \ref{sec11}.} Given $\mu\in\mathcal E^+$ and $\mathcal F\subset\mathcal E^+$, write
\begin{equation*}\label{not1}\varrho(\mu,\mathcal F):=\inf_{\nu\in\mathcal F}\,\|\mu-\nu\|.\end{equation*}

\begin{theorem}\label{th-proj}Assume that\/ $\mathcal F$ is a strongly closed, convex subset of\/ $\mathcal E^+$. For any\/ $\mu\in\mathcal E^+$, there is a unique\/ $\mu_{\mathcal F}\in\mathcal F$ having the property
\begin{equation*}\label{def-proj}\|\mu-\mu_{\mathcal F}\|=\min_{\nu\in\mathcal F}\,\|\mu-\nu\|=\varrho(\mu,\mathcal F);\end{equation*}
it is called the orthogonal projection of\/ $\mu$ in the pre-Hilbert space\/ $\mathcal E$ onto\/ $\mathcal F$. If moreover\/ $\mathcal F$ is a cone in\/ $\mathcal E^+$ that contains zero measure, then such a\/ $\mu_{\mathcal F}$ is characterized uniquely within\/ $\mathcal F$  by the two relations
\begin{align}
\label{proj1}&(\mu-\mu_{\mathcal F},\nu)\leqslant0\text{ \ for all\ }\nu\in\mathcal F,\\
\label{proj2}&(\mu-\mu_{\mathcal F},\mu_{\mathcal F})=0.
\end{align}
\end{theorem}

\begin{proof}We conclude this from \cite{E2} (Theorem~1.12.3 and Proposition~1.12.4(2)) by noting that $\mathcal F$ is complete in the induced strong topology, $\mathcal F$ being a strongly closed subset of the strongly complete cone $\mathcal E^+$. (The existence of $\mu_{\mathcal F}$ is due to the consistency of the kernel, whereas the uniqueness is implied by the energy principle.)\end{proof}

\begin{definition}[{\rm Fuglede \cite{F71}}]\label{def-quasi} A set $A\subset X$ is said to be {\it quasiclosed\/} if
\[
\inf\,\bigl\{c^*(A\bigtriangleup F):\ F\text{ closed, }F\subset X\bigr\}=0,
\]
where $\bigtriangleup$ denotes the symmetric difference.
\end{definition}

\begin{lemma}\label{l-quasi0}For\/ $A\subset X$ quasiclosed, the cone\/ $\mathfrak M^+_A$ and its truncated subcone\/
$\check{\mathfrak M}^+_A:=\{\nu\in\mathfrak M^+_A:\ \nu_*(A)\leqslant1\}$ are both vaguely closed.\end{lemma}

\begin{proof}For $A$ quasiclosed, $\mathfrak M^+_A$ is vaguely closed according to \cite[Corollary~6.2]{Fu4}. Given a vague limit point $\nu$ of $\check{\mathfrak M}^+_A$, choose a net $(\nu_k)\subset\check{\mathfrak M}^+_A$ that converges to $\nu$ vaguely. Then $\nu\in\mathfrak M^+_A$ (see above), and moreover
\[\nu_*(A)=\nu(X)\leqslant\liminf_k\,\nu_k(X)=\liminf_k\,(\nu_k)_*(A)\leqslant1,\]
the former inequality being valid by Lemma~\ref{lsc} below. Thus indeed $\nu\in\check{\mathfrak M}^+_A$.\end{proof}

\begin{lemma}\label{lsc}The mapping\/ $\mu\mapsto\mu(X)$ is vaguely l.s.c.\ on $\mathfrak M^+$.\end{lemma}

\begin{proof} This is obvious from Lemma~\ref{lemma-semi} with $\psi:=1_X$.\end{proof}

\begin{lemma}\label{l-quasi}If\/ $A\subset X$ is quasiclosed, then the cone\/ $\mathcal E^+_A$ is strongly closed; hence
\begin{equation}\label{eq-id0}\mathcal E_A^+=\mathcal E_A',\end{equation}
$\mathcal E_A'$ denoting the strong closure of\/ $\mathcal E_A^+$. The same holds if\/ $\mathcal E^+_A$ is replaced by  its truncated subcone\/
$\check{\mathcal E}^+_A:=\check{\mathfrak M}^+_A\cap\mathcal E$
while\/ $\mathcal E_A'$ by\/ $\check{\mathcal E}_A'$, the strong closure of\/ $\check{\mathcal E}^+_A$.
\end{lemma}

\begin{proof}Given $\nu\in\mathcal E'_A$, choose a net $(\nu_k)\subset\mathcal E^+_A$ that converges to $\nu$ strongly. The kernel being perfect, $\nu_k\to\nu$ also vaguely (Sect.~\ref{sec11}). As the cone $\mathfrak M^+_A$ is vaguely closed (Lemma~\ref{l-quasi0}), we have $\nu\in\mathfrak M^+_A$, and hence actually $\nu\in\mathcal E^+_A$. Since the truncated cone $\check{\mathfrak M}^+_A$ is likewise vaguely closed (Lemma~\ref{l-quasi0}), the rest of the proof runs as before.
\end{proof}

\section{Inner balayage}\label{sec-inner}

{\it From now on, a\/ {\rm(}perfect\/{\rm)} kernel\/ $\kappa$ on a l.c.\ space\/ $X$ is supposed to satisfy the domination principle.}
Fix arbitrary $\mu\in\mathcal E^+$ and $A\subset X$.

\begin{definition}\label{def-bal}$\mu^A=\mu^A_*\in\mathcal E^+$ is said to be an {\it inner balayage\/} of $\mu$ to $A$ if
\begin{equation}\label{eq-bal}\kappa\mu^A=\inf_{\nu\in\Lambda_{A,\mu}}\,\kappa\nu\text{ \ on\ }X,\end{equation}
where
\[\Lambda_{A,\mu}:=\bigl\{\nu\in\mathcal E^+: \ \kappa\nu\geqslant\kappa\mu\text{ \ n.e.\ on $A$}\bigr\}.\]
\end{definition}

As seen from \cite[Section~19, Theorem~1]{Ca2}, this definition is in agreement with Cartan's (classical) concept of inner Newtonian balayage on $\mathbb R^n$, $n\geqslant3$ (cf.\ also \cite[Theorem~4.3]{Z-bal} providing a characteristic property of inner Riesz balayage).

\begin{lemma}\label{l-unique}The inner balayage\/ $\mu^A$ is unique\/ {\rm(}if it exists\/{\rm)}.\end{lemma}

\begin{proof} If both $\theta_1,\theta_2\in\mathcal E^+$ satisfy (\ref{eq-bal}), then $\kappa\theta_1$ and $\kappa\theta_2$ take finite equal values q.e.\ on $X$, cf.\ \cite[Corollary to Lemma~3.2.3]{F1}. Applying \cite[Lemma~3.2.1(a)]{F1} therefore gives $\|\theta_1-\theta_2\|=0$, and hence $\theta_1=\theta_2$, by the energy principle.\end{proof}

\subsection{Existence of the inner balayage. Alternative definitions}\label{sec-3-1} Given $\mu\in\mathcal E^+$ and $A\subset X$, denote by $\mu_{\mathcal E_A'}$ the orthogonal projection of $\mu$ in the pre-Hil\-bert space $\mathcal E$ onto $\mathcal E_A'$, the strong closure of $\mathcal E_A^+$. Such a projection $\mu_{\mathcal E_A'}$ exists and is unique by Theorem~\ref{th-proj} applied to the (strongly closed, convex) set $\mathcal E_A'$.

\begin{theorem}\label{th-bal-ex}Given\/ $\mu\in\mathcal E^+$ and\/ $A\subset X$, the inner balayage\/ $\mu^A$, introduced by Definition\/~{\rm\ref{def-bal}}, does exist. Actually,
\begin{equation}\label{proj''}\mu^A=\mu_{\mathcal E_A'},\end{equation}
and hence the inner balayage\/ $\mu^A$ can alternatively be determined by the two formulae
\[\mu^A\in\mathcal E_A',\quad\|\mu-\mu^A\|=\min_{\nu\in\mathcal E_A'}\,\|\mu-\nu\|=\inf_{\nu\in\mathcal E_A^+}\,\|\mu-\nu\|.\]
Furthermore, $\mu^A$ has the properties
\begin{align}
\label{eq-bala-f1}\kappa\mu^A&=\kappa\mu\text{ \ n.e.\ on\ }A,\\
\label{eq-bala-f0}\kappa\mu^A&=\kappa\mu\text{ \ $\mu^A$-a.e.,}\\
\kappa\mu^A&\leqslant\kappa\mu\text{ \ on\ }X,\label{eq-bala-f2}
\end{align}
and it can equivalently be defined as the only measure in\/ $\mathcal E_A'$ satisfying\/ {\rm(\ref{eq-bala-f1})}.\footnote{That the inner balayage $\mu^A$ is uniquely determined by (\ref{eq-bala-f1}) within $\mathcal E_A'$ ($A$ being arbitrary) seems to be unknown before even for the (classical) Newtonian, $\alpha$-Riesz, and $\alpha$-Green kernels.}
\end{theorem}

For a proof of Theorem~\ref{th-bal-ex}, see Sect.~\ref{sec-pr1} below. If $A=K$ is compact, the proof is based on the classical Gauss variational method \cite{C0,Ca2}, generalized to a perfect kernel satisfying the domination principle. In order to extend the relation
\[\mu^K=\mu_{\mathcal E_K^+}=\mu_{\mathcal E_K'}\] thereby obtained to $A$ arbitrary, we shall establish convergence assertions for the nets $(\mu_{\mathcal E_K^+})$ and $(\kappa\mu_{\mathcal E_K^+})$ when $K$ increases along the upper directed family of all compact subsets of $A$ (see Eqs.~(\ref{eq-cont'}) and (\ref{later-ar}) in Theorem~\ref{th-pr}).

\begin{corollary}\label{cor-def}Variational problem\/ {\rm(\ref{eq-bal})} on minimizing the potential\/ $\kappa\nu$ among the measures\/ $\nu\in\Lambda_{A,\mu}$ has the unique solution\/ $\mu^A$, determined by Theorem\/~{\rm\ref{th-bal-ex}}.
\end{corollary}

\begin{proof}By Theorem~\ref{th-bal-ex} and Lemma~\ref{l-unique}, the infimum in (\ref{eq-bal}) is achieved at the unique measure $\mu^A$, and moreover $\mu^A\in\Lambda_{A,\mu}$, by (\ref{eq-bala-f1}). Hence the infimum in (\ref{eq-bal}) is indeed an actual minimum.
\end{proof}

The inner balayage $\mu^A$ is {\it not\/} concentrated on the set $A$ itself, but on its closure,
and this occurs even in the case of the Newtonian kernel on $\mathbb R^n$, $n\geqslant3$, and an open ball.\footnote{It may even happen that $S(\mu^A)\cap A=\varnothing$. For instance, for any $\nu\in\mathfrak M^+(\mathbb R^n)$, $n\geqslant3$, and any open set $D\subset\mathbb R^n$ with connected complement $D^c$ of nonzero Newtonian capacity, we have $S(\nu^D)=\partial_{\mathbb R^n}D$, and hence $S(\nu^D)\cap D=\varnothing$, $\nu^D$ denoting the inner Newtonian balayage of $\nu$ to $D$. This follows e.g.\ from
\cite[Theorems~4.1, 5.1]{Z-bal2}.} The inclusion $\mu^A\in\mathcal E^+_A$, however, does already hold if $A$ is quasiclosed.

\begin{corollary}\label{cor-def1}
If\/ $A$ is quasiclosed\/ {\rm(}more generally, if\/ $\mathcal E^+_A$ is strongly closed\/{\rm)}, the inner balayage\/ $\mu^A$ can be found as the orthogonal projection of\/ $\mu$ onto\/ $\mathcal E^+_A$, i.e.
\[\mu^A=\mu_{\mathcal E^+_A},\]
or alternatively, as the only measure in\/ $\mathcal E^+_A$ having property\/ {\rm(\ref{eq-bala-f1})}.
\end{corollary}

\begin{proof} For quasiclosed $A$, we have $\mathcal E^+_A=\mathcal E'_A$ (Lemma~\ref{l-quasi}), hence $\mu_{\mathcal E_A^+}=\mu_{\mathcal E_A'}$,
which substituted into
Theorem~\ref{th-bal-ex} yields the corollary.\end{proof}

Assume now that $\mu\in\mathcal E^+$ is bounded. Given $A\subset X$, denote by $\check{\mathcal E}^+_{A,\mu}$, resp.\ $\check{\Lambda}_{A,\mu}$, the (convex) class of all $\nu\in\mathcal E^+_A$, resp.\ $\nu\in\Lambda_{A,\mu}$, having the property
\begin{equation}\label{leq}\nu(X)\leqslant\mu(X),\end{equation}
and by $\check{\mathcal E}'_{A,\mu}$ the strong closure of $\check{\mathcal E}^+_{A,\mu}$.

\begin{theorem}\label{th-fr}Suppose that\/ $\mu\in\mathcal E^+$ is bounded while\/ $\kappa$ satisfies Frostman's maximum principle. Then Definition\/~{\rm\ref{def-bal}} and that obtained from it by replacing\/ $\Lambda_{A,\mu}$ by\/ $\check{\Lambda}_{A,\mu}$ lead to the same concept of inner balayage.
Furthermore, Theorem\/~{\rm\ref{th-bal-ex}} and Corollaries\/~{\rm\ref{cor-def}} and\/ {\rm\ref{cor-def1}} remain valid if\/ $\Lambda_{A,\mu}$, $\mathcal E^+_A$, and\/ $\mathcal E'_A$ are replaced throughout by\/ $\check{\Lambda}_{A,\mu}$, $\check{\mathcal E}^+_{A,\mu}$, and\/ $\check{\mathcal E}'_{A,\mu}$, respectively. Thus, under the stated assumptions,
the inner balayage\/ $\mu^A$ can equivalently be defined by either of the two formulae\/\footnote{The orthogonal projections involved in (\ref{eq-balu1}) and (\ref{eq-balu2}) do exist (Theorem~\ref{th-proj} and Lemma~\ref{l-quasi}).}
\begin{align}\kappa\mu^A&=\inf_{\nu\in\check{\Lambda}_{A,\mu}}\,\kappa\nu\text{ \ on\ }X,\label{eq-balu}\\
\mu^A&=\mu_{\check{\mathcal E}'_{A,\mu}},\label{eq-balu1}\end{align}
the infimum in\/ {\rm(\ref{eq-balu})} being an actual minimum. If\/ $A$ is quasiclosed, then also\/
\begin{equation}\label{eq-balu2}\mu^A=\mu_{\check{\mathcal E}^+_{A,\mu}}.\end{equation}
\end{theorem}

\begin{remark}\label{weakened} Theorem~\ref{th-fr} still holds if (\ref{leq}) is weakened to $\nu(X)\leqslant q\mu(X)$, where $q\in[1,\infty)$ (cf.\ Sect.~\ref{sec-pr-fr}).
\end{remark}

\subsection{Convergence theorems for inner balayage} Theorem~\ref{th-bal-cont} below shows that both $\mu^A$ and $\kappa\mu^A$ are continuous under the exhaustion of $A$ by $K$ compact, which justifies the term `inner' balayage.

\begin{theorem}\label{th-bal-cont}For\/ $\mu\in\mathcal E^+$ and\/ $A\subset X$ arbitrary,
\begin{equation}\label{eq-bal-cont}\mu^K\to\mu^A\text{ \ strongly and vaguely in $\mathcal E^+$ as $K\uparrow A$},\end{equation}
where the abbreviation $K\uparrow A$ means that\/ $K$ increases along the upper directed family\/ $\mathfrak C=\mathfrak C_A$ of all compact subsets of\/ $A$. Furthermore,
\begin{equation}\label{later-bal}\kappa\mu^K\uparrow\kappa\mu^A\text{ \ pointwise on\/ $X$ as\/ $K\uparrow A$},\end{equation}
and hence
\begin{equation}\label{later-bal1}\kappa\mu^A=\sup_{K\in\mathfrak C_A}\,\kappa\mu^K\text{ \ on\ }X.\end{equation}
\end{theorem}

\begin{theorem}\label{th-cont-bor1}Suppose that\/ $A$ is the union of an increasing sequence\/ $(A_k)$ of universally measurable subsets of\/ $X$. For any\/ $\mu\in\mathcal E^+$, then
\begin{align}\label{eq-cont-bor}&\mu^{A_k}\to\mu^A\text{ \ strongly and vaguely in $\mathcal E^+$ as $k\to\infty$},\\
&\kappa\mu^{A_k}\uparrow\kappa\mu^A\text{ \ pointwise on\/ $X$ as\/ $k\to\infty$}.\label{later-bor}
\end{align}
\end{theorem}

\begin{theorem}\label{th-cont-bor2}Let\/ $A$ be the intersection of a decreasing net\/ {\rm(}resp.\ a decreasing sequence\/{\rm)} $(A_t)_{t\in T}$ of closed\/ {\rm(}resp.\ quasiclosed\/{\rm)} sets. Given\/ $\mu\in\mathcal E^+$, the following two limit relations hold as\/ $t$ increases along\/ $T$:
\begin{align}\label{eq-cl}&\mu^{A_t}\to\mu^A\text{ \ strongly and vaguely in $\mathcal E^+$},\\
&\kappa\mu^{A_t}\downarrow\kappa\mu^A\text{ \ pointwise q.e.\ on\/ $X$}.\label{later-cl}
\end{align}
\end{theorem}

In view of the equality $\mu^A=\mu_{\mathcal E_A'}$ (see Eq.~(\ref{proj''})), the proofs of Theorems~\ref{th-bal-cont}--\ref{th-cont-bor2} (Sects.~\ref{sec-pr1}, \ref{sec-pr-bor}, \ref{sec-pr4}) are mainly based on a careful analysis of the orthogonal projection $\mu_{\mathcal E_A'}$, provided by Theorem~\ref{th-pr} below, as well as on some further properties of inner swept measures and their potentials, presented in Sect.~\ref{sec-further}.

\section{On the orthogonal projection of $\mu\in\mathcal E^+$ onto $\mathcal E_A'$}

\begin{theorem}\label{th-pr}Given\/ $\mu\in\mathcal E^+$ and\/ $A\subset X$,
\begin{align}
\label{eq-pr-1}\kappa\mu_{\mathcal E_A'}&=\kappa\mu\text{ \ n.e.\ on\ }A,\\
\label{eq-pr-0}\kappa\mu_{\mathcal E_A'}&=\kappa\mu\text{ \ $\mu_{\mathcal E_A'}$-a.e.,}\\
\kappa\mu_{\mathcal E_A'}&\leqslant\kappa\mu\text{ \ on\ }X.\label{eq-pr-2}
\end{align}
Furthermore,
\begin{align}\label{eq-cont'}&\mu_{\mathcal E_K^+}\to\mu_{\mathcal E_A'}\text{ \ strongly and vaguely in $\mathcal E^+$ as $K\uparrow A$},\\
\label{later-ar}&\kappa\mu_{\mathcal E_K^+}\uparrow\kappa\mu_{\mathcal E_A'}\text{ \ pointwise on\/ $X$ as\/ $K\uparrow A$}.\end{align}
\end{theorem}

\begin{proof} For any $K,K'\in\mathfrak C=\mathfrak C_A$ such that $K\subset K'$, $\mu_{\mathcal E_K^+}\in\mathcal E_K^+\subset\mathcal E_{K'}^+\subset\mathcal E_A^+$\,(\footnote{These two inclusions follow from the fact that any subset of a locally negligible set is likewise locally negligible \cite[Section~IV.5.2]{B2}, which will often be used in the sequel.\label{f-negl}}) and
\[\min_{\nu\in\mathcal E_{K'}^+}\,\|\mu-\nu\|=\|\mu-\mu_{\mathcal E_{K'}^+}\|\leqslant\|\mu-\mu_{\mathcal E_K^+}\|,\]
by the definition of orthogonal projection. Lemma~\ref{4.1.1} with $\Gamma:=\{\mu-\nu:\ \nu\in\mathcal E^+_{K'}\}$ and $\lambda:=\mu-\mu_{\mathcal E_{K'}^+}$ therefore yields
\[\|\mu_{\mathcal E_K^+}-\mu_{\mathcal E_{K'}^+}\|^2=\|(\mu-\mu_{\mathcal E_K^+})-(\mu-\mu_{\mathcal E_{K'}^+})\|^2\leqslant\|\mu-\mu_{\mathcal E_K^+}\|^2-\|\mu-\mu_{\mathcal E_{K'}^+}\|^2.\]
Being decreasing and lower bounded, the net $(\|\mu-\mu_{\mathcal E_K^+}\|)_{K\in\mathfrak C}$ is Cauchy in $\mathbb R$, which together with the preceding display implies that  $(\mu_{\mathcal E_K^+})_{K\in\mathfrak C}$ is strong Cauchy in $\mathcal E^+$. In view of the perfectness of the kernel, there is a unique $\mu_0\in\mathcal E^+$ such that
\[\mu_{\mathcal E_K^+}\to\mu_0\text{ \ strongly and vaguely as $K\uparrow A$},\]
and hence $\mu_0\in\mathcal E_A'$. To prove (\ref{eq-cont'}), it is thus enough to verify the equality
\begin{equation}\label{eq-f-cont'''}\|\mu-\mu_0\|=\varrho(\mu,\mathcal E_A')\quad\bigl({}:=\inf_{\nu\in\mathcal E_A'}\,\|\mu-\nu\|\bigr).\end{equation}

It follows from the above that
\begin{equation}\label{stream}\varrho(\mu,\mathcal E_A^+)=\varrho(\mu,\mathcal E_A')\leqslant\|\mu-\mu_0\|=\lim_{K\uparrow A}\,\|\mu-\mu_{\mathcal E_K^+}\|=\lim_{K\uparrow A}\,\varrho(\mu,\mathcal E_K^+),\end{equation}
the first equality being evident by definition. On the other hand,
for every $\nu\in\mathcal E^+_A$, $A$ is $\nu$-meas\-ur\-able and $\nu=\nu|_A$. Applying \cite[Lemma~1.2.2]{F1} to an arbitrary function $f$ from $C_0^+:=\{g\in C_0:\ g\geqslant0\}$ we therefore get
\[\nu(f)=\nu|_A(f)=\int f\,d\nu|_A=\sup_{K\in\mathfrak C_A}\,\int f\,d\nu|_K=\lim_{K\uparrow A}\,\nu|_K(f).\]
Thus, for every $\nu\in\mathcal E^+_A$,
\[\nu|_K\to\nu\text{ \ vaguely as $K\uparrow A$},\]
which gives, by the principle of descent,
\[\kappa(\nu,\nu)\leqslant\lim_{K\uparrow A}\,\kappa(\nu|_K,\nu|_K),\quad\kappa(\mu,\nu)\leqslant\lim_{K\uparrow A}\,\kappa(\mu,\nu|_K).\]
The opposite being obvious in view of the positivity of the kernel, equality in fact prevails in these two inequalities. So,
\[\|\mu-\nu\|=\lim_{K\uparrow A}\,\|\mu-\nu|_K\|\geqslant\lim_{K\uparrow A}\,\varrho(\mu,\mathcal E_K^+)\text{ \ for every\ }\nu\in\mathcal E^+_A,\]
and consequently
\[\varrho(\mu,\mathcal E^+_A)\geqslant\lim_{K\uparrow A}\,\varrho(\mu,\mathcal E_K^+).\]
Combining this with (\ref{stream}) establishes (\ref{eq-f-cont'''}), and hence (\ref{eq-cont'}).

The remainder of the proof is based on the fact that the orthogonal projection $\mu_{\mathcal E_A'}$ is characterized uniquely within $\mathcal E'_A$ by (\ref{proj1}) and (\ref{proj2}) with $\mathcal F:=\mathcal E'_A$, $\mathcal E'_A$ being a convex subcone of $\mathcal E^+$ which contains zero measure. If $A=K$ is compact, this actually holds with $\mathcal E^+_K$ in place of $\mathcal E_K'$, see (\ref{eq-id0}).

Assuming first that $A=K$ is compact, we begin by showing that $c_*(E)=0$, where $E:=\{x\in K:\ \kappa\mu(x)>\kappa\mu_{\mathcal E_K^+}(x)\}$. For any $\sigma\in\mathcal E^+$ and any compact $Q\subset E$, we have $\sigma|_Q\in\mathcal E^+_Q\subset\mathcal E^+_K$, and hence, by (\ref{proj1}) with $\mathcal F:=\mathcal E^+_K$,
\[(\mu-\mu_{\mathcal E_K^+},\sigma|_Q)=\int(\kappa\mu-\kappa\mu_{\mathcal E_K^+})\,d\sigma|_Q\leqslant0.\]
As seen from the definition of $E$, equality must prevail here; or, equivalently,
\[\kappa\mu-\kappa\mu_{\mathcal E_K^+}=0\text{ \ $\sigma|_Q$-a.e.}\]
(cf.\  \cite[Section~IV.2, Theorem~1]{B2}), which  is possible only if $\sigma|_Q=0$. Applying (\ref{2.3.1}), we therefore conclude that, indeed, $c_*(E)=0$, and so
\begin{equation}\label{geq1}\kappa\mu_{\mathcal E_K^+}\geqslant\kappa\mu\text{ \ n.e.\ on\ }K.\end{equation}

Since $c_*(E)=0$ while $\mu_{\mathcal E_K^+}\in\mathcal E^+_K$ is bounded, the (Borel) set $E$ is $\mu_{\mathcal E_K^+}$-neg\-lig\-ible (cf.\ Lemma~\ref{l-negl}).\footnote{The set $E$ is Borel because  $\kappa(\mu-\mu_{\mathcal E_K^+})$ is Borel measurable on $X$, being the difference between two l.s.c.\ functions.} Thus $\kappa\mu_{\mathcal E_K^+}\geqslant\kappa\mu$ $\mu_{\mathcal E_K^+}$-a.e., which in view of (\ref{proj2}) implies that
\[\kappa\mu_{\mathcal E_K^+}=\kappa\mu\text{ \ $\mu_{\mathcal E_K^+}$-a.e.}\]
By the domination principle, $\kappa\mu_{\mathcal E_K^+}\leqslant\kappa\mu$ on $X$, which combined with (\ref{geq1}) gives
\begin{equation}\label{eq-pr-c1}\kappa\mu_{\mathcal E_K^+}=\kappa\mu\text{ \ n.e.\ on\ }K.\end{equation}

Dropping now the current requirement on $A$ of being compact and applying (\ref{eq-pr-c1}) to each of $K,K'\in\mathfrak C_A$ such that $K\subset K'$, we get
\[\kappa\mu_{\mathcal E_K^+}=\kappa\mu=\kappa\mu_{\mathcal E_{K'}^+}\text{ \ n.e.\ on\ }K,\]
the inner capacity being countably subadditive on universally measurable subsets of $X$ \cite[Lemma~2.3.5]{F1}.
Similarly as in the preceding paragraph, this actually holds $\mu_{\mathcal E_K^+}$-a.e., which yields, again by the domination principle,
\[\kappa\mu_{\mathcal E_K^+}\leqslant\kappa\mu_{\mathcal E_{K'}^+}\leqslant\kappa\mu\text{ \ on\ }X.\]
Hence, for some function $\psi$ that is ${}\leqslant\kappa\mu$ on $X$,
\begin{equation}\label{convv}\kappa\mu_{\mathcal E_K^+}\uparrow\psi\text{ \ pointwise on $X$ as $K\uparrow A$}.\end{equation}
To establish both (\ref{eq-pr-2}) and (\ref{later-ar}), it thus remains to show that
\begin{equation}\label{convv''}\psi=\kappa\mu_{\mathcal E_A'}\text{ \ on $X$}.\end{equation}

The strong topology on $\mathcal E$ having a countable base of neighborhoods, it follows from (\ref{eq-cont'}) that there is a subsequence $(\mu_{\mathcal E_{K_j}^+})_{j\in\mathbb N}$ of the net $(\mu_{\mathcal E_K^+})_{K\in\mathfrak C}$ which converges strongly (hence vaguely) to $\mu_{\mathcal E_A'}$. Therefore, by \cite[Theorem~3.6]{Fu5},
\[\kappa\mu_{\mathcal E_{K_j}^+}\to\kappa\mu_{\mathcal E_A'}\text{ \ pointwise q.e.\ on $X$ (as $j\to\infty$)},\]
which in view of (\ref{convv}) implies that $\psi=\kappa\mu_{\mathcal E_A'}$ q.e.\ on $X$.
Thus, for every $K\in\mathfrak C_A$, the inequality
\[\kappa\mu_{\mathcal E_K^+}\leqslant\kappa\mu_{\mathcal E_A'}\]
holds q.e.\ on $X$, hence $\mu_{\mathcal E_K^+}$-a.e., and so everywhere on $X$, the last two conclusions being obtained in the same manner as above. Letting now $K\uparrow A$ gives
\[\psi=\limsup_{K\uparrow A}\,\kappa\mu_{\mathcal E_K^+}\leqslant\kappa\mu_{\mathcal E_A'}\leqslant\liminf_{K\uparrow A}\,\kappa\mu_{\mathcal E_K^+}\text{ \ on\ }X,\]
the latter inequality being derived from (\ref{eq-cont'}) by use of the principle of descent. This establishes (\ref{convv''}), and hence also both (\ref{eq-pr-2}) and (\ref{later-ar}).

The proof is completed by verifying (\ref{eq-pr-1}) and (\ref{eq-pr-0}). The function $\varphi:=\kappa\mu-\kappa\mu_{\mathcal E_A'}$ is $\nu$-neg\-lig\-ible for every $\nu\in\mathcal E'_A$ (cf.\ \cite[Section~IV.2, Definition~1]{B2}), for
\[\int\varphi\,d\nu=(\mu-\mu_{\mathcal E_A'},\nu)\leqslant0\]
by (\ref{proj1}) with $\mathcal F:=\mathcal E_A'$, whereas $\varphi\geqslant0$ by (\ref{eq-pr-2}). Again by \cite[Section~IV.2, Theorem~1]{B2}, 
\begin{equation}\label{proj-proof}\kappa\mu_{\mathcal E_A'}=\kappa\mu\text{ \ $\nu$-a.e.\ for every\ }\nu\in\mathcal E'_A,\end{equation}
which leads to (\ref{eq-pr-1}) in a manner similar to that in the proof of (\ref{eq-pr-c1}). Indeed, for every $\sigma\in\mathcal E^+$ and every compact subset $Q$ of the set $E_1$ of all $x\in A$ with $\kappa\mu_{\mathcal E_A'}(x)<\kappa\mu(x)$, we have $\sigma|_Q=0$, hence $\mathcal E^+_Q=\{0\}$, and consequently $c_*(E_1)=0$, by (\ref{2.3.1}). Finally, (\ref{proj-proof}) with $\nu:=\mu_{\mathcal E_A'}$ $({}\in\mathcal E'_A$) yields (\ref{eq-pr-0}).
\end{proof}

\begin{remark}\label{rem-char}For a compact set $A$, the orthogonal projection $\mu_{\mathcal E_A'}$  $({}=\mu_{\mathcal E_A^+})$ is characterized uniquely within $\mathcal E'_A$ $({}=\mathcal E_A^+)$ by (\ref{eq-pr-1}).\footnote{The same actually holds if $A$ is closed or even quasiclosed, which is seen from Corollary~\ref{cor-def1}. Moreover, this can be generalized to arbitrary $A$ as follows: {\it the orthogonal projection\/ $\mu_{\mathcal E_A'}$ is characterized uniquely within\/ $\mathcal E'_A$ by\/ {\rm(\ref{eq-pr-1})}}. Indeed, on account of (\ref{proj''}), such an extension is derived directly from the characteristic property of inner balayage given in Theorem~\ref{th-bal-ex}.} Indeed, if $\kappa\theta=\kappa\mu$ n.e.\ on $A$ for some $\theta\in\mathcal E^+_A$, then, by (\ref{eq-pr-1}), the equality $\kappa\theta=\kappa\mu_{\mathcal E_A^+}$ holds true n.e.\ on $A$, and hence $(\theta+\mu_{\mathcal E_A^+})$-a.e.\ (Lemma~\ref{l-negl}), the set $\{x\in A:\ \kappa\theta(x)\ne\kappa\mu_{\mathcal E_A^+}(x)\}$ being Borel measurable while the measure $\theta+\mu_{\mathcal E_A^+}\in\mathcal E^+_A$ being bounded. This implies by integration $\|\theta-\mu_{\mathcal E_A^+}\|=0$, and consequently $\theta=\mu_{\mathcal E_A^+}$, by the energy principle.\end{remark}

\section{Proofs of Theorems~\ref{th-bal-ex} and \ref{th-bal-cont}}\label{sec-pr1}
Due to Theorem~\ref{th-pr}, we are now able to establish Theorems~\ref{th-bal-ex} and \ref{th-bal-cont}.

Fix $\mu\in\mathcal E^+$ and $A\subset X$. Noting from (\ref{eq-pr-1}) that $\mu_{\mathcal E_A'}\in\Lambda_{A,\mu}$ (see Definition~\ref{def-bal}), we shall verify (\ref{proj''}) by proving that for any given $\nu\in\Lambda_{A,\mu}$,
\begin{equation}\label{eq-char1}\kappa\mu_{\mathcal E_A'}\leqslant\kappa\nu\text{ \ on $X$}.\end{equation}
Combining $\kappa\nu\geqslant\kappa\mu$ n.e.\ on $A$ with (\ref{eq-pr-c1}) applied to any $K\in\mathfrak C_A$ shows that
\[\kappa\mu_{\mathcal E_K^+}\leqslant\kappa\nu\]
holds n.e.\ on $K$, hence $\mu_{\mathcal E^+_K}$-a.e.\ by Lemma~\ref{l-negl},\footnote{The application of Lemma~\ref{l-negl} is justified in the same manner as above, see e.g.\ Remark~\ref{rem-char}.} and consequently on all of $X$, by the domination principle. On account of (\ref{later-ar}), letting here $K\uparrow A$ results in (\ref{eq-char1}), and (\ref{proj''}) follows.
Substituting now (\ref{proj''}) into (\ref{eq-pr-1})--(\ref{eq-pr-2}) gives (\ref{eq-bala-f1})--(\ref{eq-bala-f2}). Also note that, because of (\ref{proj''}),
\begin{equation}\label{cor-st}\mu^A=\mu\text{ \ for every\ }\mu\in\mathcal E'_A.\end{equation}

By (\ref{proj''}) and (\ref{eq-id0}) with $A:=K$ compact, $\mu^K=\mu_{\mathcal E_K^+}$. Combining this and (\ref{proj''}) with (\ref{eq-cont'}) and (\ref{later-ar}) yields (\ref{eq-bal-cont}) and (\ref{later-bal}), thereby establishing Theorem~\ref{th-bal-cont}.

To complete the proof of Theorem~\ref{th-bal-ex}, it remains to verify that $\mu^A$ is characterized uniquely within $\mathcal E_A'$ by (\ref{eq-bala-f1}). Having fixed  $\theta\in\mathcal E_A'$ with the property
\begin{equation}\label{eq-pot}\kappa\theta=\kappa\mu\text{ \ n.e.\ on\ }A,\end{equation}
we need to show that then necessarily
\begin{equation}\label{char}\theta=\mu^A.\end{equation}
Applying (\ref{eq-pr-c1}) to $K\in\mathfrak C_A$ and each of $\theta$ and $\mu$, we conclude from (\ref{eq-pot}) that
\[\kappa\theta_{\mathcal E_K^+}=\kappa\theta=\kappa\mu=\kappa\mu_{\mathcal E_K^+}\text{ \ n.e.\ on $K$}.\]
In view of the characteristic property of orthogonal projection onto $\mathcal E_K^+$ observed in Remark~\ref{rem-char}, this gives $\theta_{\mathcal E_K^+}=\mu_{\mathcal E_K^+}$, and hence, by (\ref{proj''}),
\[\theta^K=\mu^K\text{ \ for every\ }K\in\mathfrak C_A.\]
Letting now $K\uparrow A$ and using (\ref{eq-bal-cont}) we get
\[\theta^A=\mu^A,\]
the vague limit of any convergent net in the (Hausdorff) space $\mathfrak M$ being unique. As $\theta^A=\theta$ by (\ref{cor-st}) applied to $\theta\in\mathcal E'_A$, (\ref{char}) follows.

\section{Further properties of inner balayage}\label{sec-further}

\begin{proposition}[{\rm Monotonicity property}]\label{cor-mon}If\/ $A\subset Q$, then for any\/ $\mu\in\mathcal E^+$,
\begin{equation*}\label{mon}\kappa\mu^A\leqslant\kappa\mu^Q.\end{equation*}
\end{proposition}

\begin{proof}Since $\mathfrak C_A\subset\mathfrak C_Q$, this follows directly from (\ref{later-bal1}).
\end{proof}

\begin{proposition}[{\rm Balayage with a rest}]\label{cor-rest}If\/ $A\subset Q$, then for any\/ $\mu\in\mathcal E^+$,
\begin{equation*}\label{eq-rest}\mu^A=(\mu^Q)^A.\end{equation*}
\end{proposition}

\begin{proof}As seen from Theorem~\ref{th-bal-ex}, both $\mu^A$ and $(\mu^Q)^A$ belong to $\mathcal E'_A$, and moreover\footnote{Eq.\ (\ref{eq-restt}) is implied by the following strengthened version of countable subadditivity for inner capacity \cite[p.~158, Remark]{F1}:
{\it For arbitrary\/ $A\subset X$ and universally measurable\/ $U_k\subset X$, $k\in\mathbb N$,} 
\[c_*\Bigl(\bigcup_{k\in\mathbb N}\,A\cap U_k\Bigr)\leqslant\sum_{k\in\mathbb N}\,c_*(A\cap U_k).\]}
\begin{equation}\label{eq-restt}\kappa(\mu^Q)^A=\kappa\mu^Q=\kappa\mu=\kappa\mu^A\text{ \ n.e.\ on $A$}.\end{equation}
By the characteristic property of inner balayage given in Theorem~\ref{th-bal-ex}, $\mu^A$ and $(\mu^Q)^A$ are indeed equal.\end{proof}

\begin{proposition}\label{fixx} Given\/ $A\subset X$,
\[\mu^A=\mu\text{ \ for every\ }\mu\in\mathcal E'_A\]
{\rm(}thus in particular for every\/ $\mu\in\mathcal E^+_A${\rm)}.\end{proposition}

\begin{proof}Since $\mu_{\mathcal E'_A}=\mu$ for every $\mu\in\mathcal E'_A$, this is obvious from (\ref{proj''}).\end{proof}

{\it In the remainder of this section, the kernel\/ $\kappa$ is supposed to satisfy Frostman's maximum principle\/} (see footnote~\ref{F} for details).

\begin{proposition}[{\rm Principle of positivity of mass}]\label{cor-mass} Given\/ $\mu\in\mathcal E^+$ and\/ $A\subset X$,
\begin{equation}\label{eq-mass}\mu^A(X)\leqslant\mu(X).\end{equation}
\end{proposition}

\begin{proof}For any $K\in\mathfrak C_A$, consider the equilibrium measure $\gamma_K$ on $K$ (Sect.~\ref{sec11}). Then, by Fubini's theorem, 
\begin{align}\label{line}\mu^K(X)&=\int1\,d\mu^K=\int\kappa\gamma_K\,d\mu^K\\{}&=\int\kappa\mu^K\,d\gamma_K=\int\kappa\mu\,d\gamma_K=\int\kappa\gamma_K\,d\mu,\notag\end{align}
because $\kappa\gamma_K=1$, resp.\ $\kappa\mu^K=\kappa\mu$, holds n.e.\ on $K$ (cf.\ (\ref{eqq}) and (\ref{eq-bala-f1})), and hence $\mu^K$-a.e., resp.\ $\gamma_K$-a.e.\ (Lemma~\ref{l-negl}). Since $\kappa\gamma_K\leqslant1$ on $X$, this gives
\[\mu^K(X)\leqslant\mu(X).\]
Letting $K\uparrow A$ we get (\ref{eq-mass}) by use of Lemma~\ref{lsc} (noting that $\mu^K\to\mu^A$ vaguely).
\end{proof}

Before providing a formula for evaluation of the total mass $\mu^A(X)$ for the inner swept measure $\mu^A$ (Proposition~\ref{th-tot}), we shall first analyze the continuity of the inner equilibrium potential $\kappa\gamma_A$ under the exhaustion of $A$ by $K$ compact.

\begin{lemma}\label{l-eq}Given\/ $A\subset X$ with\/ $c_*(A)<\infty$,
\begin{equation}\label{later-eq}\kappa\gamma_K\uparrow\kappa\gamma_A\text{ \ pointwise on\/ $X$ as\/ $K\uparrow A$},\end{equation}
$\gamma_A$ denoting the inner equilibrium measure for\/ $A$.
\end{lemma}

\begin{proof}Since $c(K)\uparrow c_*(A)$ as $K\uparrow A$ (cf.\ (\ref{153})), we conclude in much the same way as it was done in \cite[Proof of Theorem~4.1]{F1} that
\begin{equation}\label{conv-eq-m}\gamma_K\to\gamma_A\text{ \ strongly (hence, vaguely) in $\mathcal E^+$ as $K\uparrow A$}.\end{equation}
But for any $K,K'\in\mathfrak C_A$ such that $K\subset K'$,
\[1=\kappa\gamma_K=\kappa\gamma_{K'}=\kappa\gamma_A\text{ \ n.e.\ on\ }K,\]
hence $\gamma_K$-a.e. Therefore, by the domination principle, the net $(\kappa\gamma_K)_{K\in\mathfrak C_A}$ increases pointwise on $X$ to some function that does not exceed  $\kappa\gamma_A$. To establish (\ref{later-eq}), we thus only need to verify the inequality
\[\kappa\gamma_A\leqslant\lim_{K\uparrow A}\,\kappa\gamma_K\text{ \ on\ }X,\]
which follows directly from (\ref{conv-eq-m}) in view of the vague lower semicontinuity of the mapping $\nu\mapsto\kappa\nu(\cdot)$, $\nu\in\mathfrak M^+$ (the principle of descent \cite[Lemma~2.2.1(b)]{F1}).
\end{proof}

\begin{proposition}\label{th-tot}For any\/ $\mu\in\mathcal E^+$ and any\/ $A\subset X$ with\/ $c_*(A)<\infty$,\footnote{For the $\alpha$-Riesz kernel of order $\alpha\in(0,2]$ on $\mathbb R^n$, $n\geqslant3$, Proposition~\ref{th-tot} remains valid for any $A\subset\mathbb R^n$ that is inner $\alpha$-thin at infinity (even if $c_*(A)=\infty$), which is seen by combining Theorems~2.5 and 5.1 from \cite{Z-bal2}. (For the concept of inner $\alpha$-thin\-ness at infinity, see \cite[Definition~2.2]{Z-bal2}.) Moreover, then the requirement on $\mu$ of having finite energy can be simply omitted.}
\begin{equation}\label{eq-mass1}\mu^A(X)=\int\kappa\gamma_A\,d\mu.\end{equation}
\end{proposition}

\begin{proof} Combining Propositions~\ref{cor-rest} and \ref{cor-mass} yields
\[\mu^K(X)=(\mu^A)^K(X)\leqslant\mu^A(X)\text{ \  for every\ }K\in\mathfrak C_A,\]
hence
\[
\limsup_{K\uparrow A}\,\mu^K(X)\leqslant\mu^A(X)\leqslant\liminf_{K\uparrow A}\,\mu^K(X),\]
the latter inequality being obtained from (\ref{eq-bal-cont}) with the aid of Lemma~\ref{lsc}. On account of (\ref{line}), this gives
\[\mu^A(X)=\lim_{K\uparrow A}\,\mu^K(X)=\lim_{K\uparrow A}\,\int\kappa\gamma_K\,d\mu.\]
Since the net $(\kappa\gamma_K)_{K\in\mathfrak C_A}$ increases pointwise on $X$ to $\kappa\gamma_A$ (Lemma~\ref{l-eq}), we get (\ref{eq-mass1}) by applying \cite[Section~IV.1, Theorem~1]{B2} to the integral on the right.
\end{proof}

\section{Proofs of Theorems~\ref{th-fr}, \ref{th-cont-bor1}, and \ref{th-cont-bor2}}\label{proofs2}

\subsection{Proof of Theorem~\ref{th-fr}}\label{sec-pr-fr} Fix $\mu\in\mathcal E^+$ and $A\subset X$. According to Theorems~\ref{th-bal-ex} and \ref{th-bal-cont}, the inner balayage $\mu^A$, introduced by Definition~\ref{def-bal}, can be found as the orthogonal projection $\mu_{\mathcal E'_A}$ of $\mu$ onto the cone $\mathcal E'_A$, and moreover
\[\mu^K\to\mu^A\text{ \ strongly (hence vaguely) as $K\uparrow A$}.\]

Assume now that $\mu$ is bounded while $\kappa$ satisfies Frostman's maximum principle. By Proposition~\ref{cor-mass},
$\mu^K(X)\leqslant\mu(X)$ for all $K\in\mathfrak C_A$, hence $\mu^K\in\check{\mathcal E}^+_{K,\mu}\subset\check{\mathcal E}^+_{A,\mu}$ (cf.\ footnote~\ref{f-negl}), 
which in view of the preceding display implies that, actually,\footnote{For the notation used here, see the paragraph followed by Theorem~\ref{th-fr}.}
\begin{equation}\label{in1}\mu^A\in\check{\mathcal E}'_{A,\mu}.\end{equation}
Noting that $\check{\mathcal E}'_{A,\mu}\subset\mathcal E'_A$, we therefore get
\[\varrho(\mu,\check{\mathcal E}'_{A,\mu})\leqslant\|\mu-\mu^A\|=\varrho(\mu,\mathcal E'_A)\leqslant\varrho(\mu,\check{\mathcal E}'_{A,\mu}),\]
hence
\[\varrho(\mu,\mathcal E'_A)=\varrho(\mu,\check{\mathcal E}'_{A,\mu}),\]
and so $\mu^A$ (${}=\mu_{\mathcal E'_A}$) serves simultaneously as the orthogonal projection of $\mu$ onto the (strongly closed, convex) truncated cone $\check{\mathcal E}'_{A,\mu}$. This proves (\ref{eq-balu1}).

As $\mu^A\in\Lambda_{A,\mu}$ (Corollary~\ref{cor-def}) and $\mu^A(X)\leqslant\mu(X)$ (Proposition~\ref{cor-mass}), 
\begin{equation}\label{in2}\mu^A\in\check{\Lambda}_{A,\mu}.\end{equation}
Thus
\[\inf_{\nu\in\check{\Lambda}_{A,\mu}}\,\kappa\nu\leqslant\kappa\mu^A=\min_{\nu\in\Lambda_{A,\mu}}\,\kappa\nu\leqslant\inf_{\nu\in\check{\Lambda}_{A,\mu}}\,\kappa\nu,\]
and so $\mu^A$ gives indeed a solution to problem (\ref{eq-balu}).

If $A$ now is quasiclosed, then the (convex) set $\check{\mathcal E}^+_{A,\mu}$ is strongly closed, which is seen from the latter part of Lemma~\ref{l-quasi} by homogeneity reasons. Hence, the orthogonal projection $\mu_{\check{\mathcal E}^+_{A,\mu}}$ does exist (see the former part of Theorem~\ref{th-proj}),\footnote{Observe that the latter part of Theorem~\ref{th-proj} is not applicable to $\mu_{\check{\mathcal E}^+_{A,\mu}}$, for $\check{\mathcal E}^+_{A,\mu}$ is not a cone.} and moreover it equals $\mu_{\check{\mathcal E}'_{A,\mu}}$. Substituting this into (\ref{eq-balu1}) establishes (\ref{eq-balu2}).

\subsection{Proof of Theorem~\ref{th-cont-bor1}}\label{sec-pr-bor} Fix $\mu\in\mathcal E^+$, and assume that $A$ is the union of an increasing sequence $(A_k)$. Then
$\mathcal E^+_{A_k}\subset\mathcal E^+_{A_{k+p}}\subset\mathcal E^+_A$ for any $k,p\in\mathbb N$ (cf.\ footnote~\ref{f-negl}), hence $\mathcal E'_{A_k}\subset\mathcal E'_{A_{k+p}}\subset\mathcal E'_A$. As $\mu^{A_k}=\mu_{\mathcal E_{A_k}'}$, we therefore have $\mu^{A_k}\in\mathcal E'_{A_k}\subset\mathcal E'_{A_{k+p}}$ and
\[\|\mu-\mu^{A_k}\|=\min_{\nu\in\mathcal E_{A_k}'}\,\|\mu-\nu\|\geqslant\min_{\nu\in\mathcal E_{A_{k+p}}'}\,\|\mu-\nu\|=\|\mu-\mu^{A_{k+p}}\|.\]
Thus, by Lemma~\ref{4.1.1} with $\Gamma:=\{\mu-\nu:\ \nu\in\mathcal E_{A_{k+p}}'\}$ and $\lambda:=\mu-\mu^{A_{k+p}}$,
\[\|\mu^{A_k}-\mu^{A_{k+p}}\|^2=\|(\mu-\mu^{A_k})-(\mu-\mu^{A_{k+p}})\|^2\leqslant\|\mu-\mu^{A_k}\|^2-\|\mu-\mu^{A_{k+p}}\|^2,\]
which together with the preceding display implies that the sequence $(\mu^{A_k})\subset\mathcal E'_A$ is strong Cauchy. Since $\mathcal E'_A$ is strongly complete, there is a unique $\mu_0\in\mathcal E'_A$ such that
\begin{equation}\label{c}\mu^{A_k}\to\mu_0\text{ \ strongly and vaguely}.\end{equation}
In turn, all this gives
\begin{align}\varrho(\mu,\mathcal E^+_A)=\varrho(\mu,\mathcal E'_A)&\leqslant\|\mu-\mu_0\|=\lim_{k\to\infty}\,\|\mu-\mu^{A_k}\|\label{line1}\\
{}&=\lim_{k\to\infty}\,\varrho(\mu,\mathcal E'_{A_k})=\lim_{k\to\infty}\,\varrho(\mu,\mathcal E^+_{A_k}).\notag\end{align}

Assuming now the sets $A_k$ to be universally measurable, we shall show that then
\begin{equation}\label{chain1}\lim_{k\to\infty}\,\varrho(\mu,\mathcal E^+_{A_k})\leqslant\varrho(\mu,\mathcal E^+_A),\end{equation}
which combined with (\ref{line1}) will imply $\|\mu-\mu_0\|=\varrho(\mu,\mathcal E'_A)$,
hence $\mu_0=\mu_{\mathcal E'_A}=\mu^A$, and consequently (\ref{eq-cont-bor}), by (\ref{c}).

In fact, for every $\nu\in\mathcal E_A^+$ and every $f\in C_0^+$,
\[\lim_{k\to\infty}\,\nu|_{A_k}(f)=\lim_{k\to\infty}\,\int1_{A_k}f\,d\nu=\int1_Af\,d\nu=\nu(f),\]
where the first and last equalities hold by the definition of the trace of a positive Radon measure to a meas\-ur\-able set (also noting that $\nu|_A=\nu$, $\nu$ being concentrated on $A$), while the second equality is obtained by applying \cite[Section~IV.1, Theorem~3]{B2} to the positive functions $1_{A_k}f$, $k\in\mathbb N$, with the upper envelope $1_Af$.
Thus
\[\nu|_{A_k}\to\nu\text{ \ vaguely},\]
which yields, by the principle of descent,
\[\kappa(\nu,\nu)\leqslant\lim_{k\to\infty}\,\kappa(\nu|_{A_k},\nu|_{A_k}),\quad\kappa(\mu,\nu)\leqslant\lim_{k\to\infty}\,\kappa(\mu,\nu|_{A_k}).\]
Since the kernel is positive, equality prevails in these two inequalities; therefore,
\begin{equation}\label{ccc}\|\mu-\nu\|=\lim_{k\to\infty}\,\|\mu-\nu|_{A_k}\|\geqslant\lim_{k\to\infty}\,\varrho(\mu,\mathcal E^+_{A_k})\text{ \ for every $\nu\in\mathcal E^+_A$},\end{equation}
and (\ref{chain1}) follows. (The inequality in (\ref{ccc}) is valid because for any $\sigma\in\mathcal E^+$ and any $\sigma$-meas\-ur\-able set $Q\subset X$, we have $\sigma|_Q\in\mathcal E^+_Q$.)

Having thus established (\ref{eq-cont-bor}), we complete the proof by verifying (\ref{later-bor}). By the monotonicity of inner balayage (Proposition~\ref{cor-mon}), the sequence $(\kappa\mu^{A_k})$ increases pointwise on $X$, and moreover
\[\lim_{k\to\infty}\,\kappa\mu^{A_k}\leqslant\kappa\mu^A\text{ \ on $X$}.\]
As $\mu^{A_k}\to\mu^A$ vaguely, the opposite inequality holds by the principle of descent.

\subsection{Proof of Theorem~\ref{th-cont-bor2}}\label{sec-pr4} Fix $\mu\in\mathcal E^+$. We are based on the fact (see Corollary~\ref{cor-def1}) that for $F\subset X$ quasiclosed (thus in particular closed),
\[\mu^F=\mu_{\mathcal E^+_F},\]
the convex cone $\mathcal E^+_F$ being strongly closed (Lemma~\ref{l-quasi}).

Suppose first that $A$ is the intersection of a lower directed family $(A_t)_{t\in T}$ of closed sets.
In view of the monotonicity of $(\mathcal E^+_{A_t})_{t\in T}$, we see in a manner similar to that in Sect.~\ref{sec-pr-bor} that $(\mu^{A_t})_{t\in T}$ is a strong Cauchy net in $\mathcal E^+$, and hence there is a unique $\mu_0\in\mathcal E^+$ with the property
\begin{equation}\label{bigcap1}\mu^{A_t}\to\mu_0\text{ \ strongly and vaguely as $t$ increases along $T$}.\end{equation}
Such a limit $\mu_0$ belongs to the class $\mathcal E^+_{A_t}$ for every $t\in T$, $\mathcal E^+_{A_t}$ being strongly closed. Since for a closed set $F\subset X$, $\mathcal E^+_F$ consists of all $\nu\in\mathcal E^+$ supported by $F$ (see Sect.~\ref{sec11}), $\mu_0$ is supported by every $A_t$, and hence by the intersection of $A_t$ over all $t$. Thus
\begin{equation}\label{bigcap}\mu_0\in\mathcal E^+_A,\end{equation}
and consequently
\[\varrho(\mu,\mathcal E^+_A)\leqslant\|\mu-\mu_0\|=\lim_{t}\,\|\mu-\mu_{\mathcal E^+_{A_t}}\|=\lim_{t}\,\varrho(\mu,\mathcal E^+_{A_t})\leqslant\varrho(\mu,\mathcal E^+_A),\]
the latter inequality being valid because $\mathcal E^+_A\subset\mathcal E^+_{A_t}$ for each $t\in T$.
This implies that, actually, $\mu_0=\mu_{\mathcal E^+_A}$, and (\ref{eq-cl}) follows.

By Proposition~\ref{cor-mon}, the net $(\kappa\mu^{A_t})_{t\in T}$ decreases pointwise on $X$, and moreover
\begin{equation}\label{p}\kappa\mu^A(x)\leqslant\lim_{t}\,\kappa\mu^{A_t}(x)\text{ \ for all $x\in X$}.\end{equation}
The strong topology on $\mathcal E$ having a countable base of neighborhoods, it follows from (\ref{eq-cl}) that there is a subsequence $(\mu^{A_{t_j}})_{j\in\mathbb N}$ of the net $(\mu^{A_t})_{t\in T}$ that converges strongly (hence vaguely) to $\mu^A$. Applying \cite[Theorem~3.6]{Fu5} we therefore conclude that equality, in fact, prevails in (\ref{p}) for quasi all $x\in X$, which establishes (\ref{later-cl}). 

Let $A$ now be the intersection of a decreasing sequence $(A_t)_{t\in T}$ of quasiclosed sets. 
In view of the fact that a countable intersection of quasiclosed sets is likewise quasiclosed
\cite[Lemma~2.3]{F71}, the proof of (\ref{eq-cl}) and (\ref{later-cl}) is essentially the same as above, the only difference being in that of (\ref{bigcap}).
As $\mu_0\in\mathcal E^+_{A_t}$ (see above), each $(A_t)^c$ is locally $\mu_0$-neg\-lig\-ible.
Being thus a countable union of locally $\mu_0$-neg\-lig\-ible sets, $A^c$ is likewise $\mu_0$-neg\-lig\-ible \cite[Section~IV.5.2]{B3}, and hence (\ref{bigcap}) indeed holds.

\section{Outer balayage}\label{sec-outer}
The approach to balayage problems, utilized in the present paper, is mainly based on an analysis of convergence of inner swept measures and their potentials under the exhaustion of $A\subset X$ by $K$ compact. We shall now show that, under suitable topological assumptions on $X$ and $A$, this (typically inner) approach is still efficient when dealing with outer balayage problems.

\begin{definition}\label{def-bal-ou}$\mu^{*A}\in\mathcal E^+$ is said to be an {\it outer balayage\/} of $\mu\in\mathcal E^+$ to $A$ if
\begin{equation}\label{eq-bal-ou}\kappa\mu^{*A}=\inf_{\nu\in\Lambda_{A,\mu}^*}\,\kappa\nu\text{ \ on\ }X,\end{equation}
where
\[\Lambda_{A,\mu}^*:=\bigl\{\nu\in\mathcal E^+: \ \kappa\nu\geqslant\kappa\mu\text{ \ q.e.\ on $A$}\bigr\}.\]
\end{definition}

By the same proof as in Lemma~\ref{l-unique}, the outer balayage $\mu^{*A}$ is {\it unique\/} (if it exists). Observe also that this definition is in agreement with Cartan's (classical) concept of outer Newtonian balayage on $\mathbb R^n$, $n\geqslant3$, cf.\ \cite[Section~19, Theorem~1$'$]{Ca2}.

\begin{definition}[{\rm Fuglede \cite{F71}}]\label{def-quasi2} A set $B\subset X$ is said to {\it quasicontain\/} a set $A\subset X$ if $c^*(A\setminus B)=0$. Two sets $A,B\subset X$ are said to be {\it $c^*$-equ\-iv\-al\-ent\/} if $c^*(A\bigtriangleup B)=0$.\end{definition}

\begin{remark}\label{Q}If the outer balayage $\mu^{*A}$ exists, then so does $\mu^{*B}$ for any $B\subset X$ that is $c^*$-equ\-iv\-al\-ent to $A$, and moreover $\mu^{*A}=\mu^{*B}$. Indeed, then $\Lambda_{A,\mu}^*=\Lambda_{B,\mu}^*$, the outer capacity being countably subadditive on any subsets of $X$ \cite[Lemma~2.3.5]{F1}.\end{remark}

{\it Unless explicitly stated otherwise, throughout Sect.\/~{\rm\ref{sec-outer}} we shall tacitly assume that a l.c.\ space\/ $X$ is perfectly normal\/\footnote{By Urysohn's theorem \cite[Section~IX.1, Theorem~1]{B3}, a topological Hausdorff space $Y$ is said to be {\it normal\/} if for any two disjoint closed subsets $F_1,F_2$ of $Y$, there exist disjoint open sets $D_1,D_2$ such that $F_i\subset D_i$ $(i=1,2)$. Further, a normal space $Y$ is said to be {\it perfectly normal\/} \cite[Section~IX.4, Exercise~7]{B3} if each closed subset of $Y$ is a countable intersection of open sets (or, equivalently, if each open subset of $Y$ is a countable union of closed sets).}  and of class\/ $K_\sigma$, and that\/ $A\subset X$ is Borel.} Suppose as before that the kernel is perfect and satisfies the domination principle.

\subsection{Existence of the outer balayage. Alternative definitions}\label{sec-ou-alt}  Fix $\mu\in\mathcal E^+$.

\begin{theorem}\label{cor-outer} The outer balayage\/ $\mu^{*A}$, introduced by Definition\/~{\rm\ref{def-bal-ou}}, does exist, and moreover it coincides with the inner balayage\/ $\mu_*^A$, introduced by Definition\/~{\rm\ref{def-bal}}. Thus, by Theorem\/~{\rm\ref{th-bal-ex}},
\begin{equation}\label{proj-outer}\mu^{*A}=\mu^A_*=\mu_{\mathcal E_A'},\end{equation}
and hence the outer balayage\/ $\mu^{*A}$ can equivalently be determined by the two formulae
\begin{equation*}\label{proj-outer'}\mu^{*A}\in\mathcal E_A',\quad\|\mu-\mu^{*A}\|=\min_{\nu\in\mathcal E_A'}\,\|\mu-\nu\|=\inf_{\nu\in\mathcal E_A^+}\,\|\mu-\nu\|.\end{equation*}
Furthermore, $\mu^{*A}$ has the properties
\begin{align}
\kappa\mu^{*A}&=\kappa\mu\text{ \ q.e.\ on\ }A,\label{eq-outer}\\
\kappa\mu^{*A}&=\kappa\mu\text{ \ $\mu^{*A}$-a.e.,}\label{eq-outer1}\\
\kappa\mu^{*A}&\leqslant\kappa\mu\text{ \ on\ }X,\label{eq-outer2}
\end{align}
and it can equivalently be defined as the only measure in\/ $\mathcal E_A'$ satisfying\/ {\rm(\ref{eq-outer})}.
\end{theorem}

\begin{corollary}\label{cor-def11}Variational problem\/ {\rm(\ref{eq-bal-ou})} on minimizing the potential\/ $\kappa\nu$ among the measures\/ $\nu\in\Lambda_{A,\mu}^*$ has the unique solution\/ $\mu^{*A}$, given by Theorem\/~{\rm\ref{cor-outer}}.
\end{corollary}

\begin{proof} According to Theorem~\ref{cor-outer}, the infimum in (\ref{eq-bal-ou}) is achieved at the (unique) measure $\mu^{*A}$, determined for instance by (\ref{proj-outer}), and moreover $\mu^{*A}\in\Lambda^*_{A,\mu}$, by (\ref{eq-outer}). Hence the infimum in (\ref{eq-bal-ou}) is indeed an actual minimum.
\end{proof}

\begin{corollary}\label{cor-def2}
Assume that a\/ {\rm(}Borel\/{\rm)} set\/ $A$ is quasiclosed\/ {\rm(}or, more generally, that\/ $\mathcal E^+_A$ is strongly closed\/{\rm)}. Then the outer balayage\/ $\mu^{*A}$ is actually the orthogonal projection of\/ $\mu$ onto the cone\/ $\mathcal E^+_A$, i.e.
\[\mu^{*A}=\mu_{\mathcal E^+_A}.\]
Alternatively, $\mu^{*A}$ can be found as the only measure in\/ $\mathcal E^+_A$ having property\/~{\rm(\ref{eq-outer})}.
\end{corollary}

\begin{remark}\label{rem-quasi}If $A$ is quasiclosed while $X$ arbitrary, the existence of $\mu^{*A}\in\mathcal E^+_A$, determined uniquely within $\mathcal E^+_A$ by $\kappa\mu^{*A}=\kappa\mu$ q.e.\ on $A$, has been established before by Fuglede \cite[Theorem~4.12]{Fu5}. Theorem~\ref{cor-outer} shows that this result, suitably modified, remains valid for any Borel subset $A$ of a perfectly normal, l.c.\ space $X$ of class $K_\sigma$, thereby presenting a further development of Fuglede's theory on outer balayage.
See also Theorems~\ref{th-ou-fr}, \ref{th-outer-cont}, and \ref{th-outer-further} below providing some additional properties of $\mu^{*A}$, which seem to be new in part even for quasiclosed $A$.
\end{remark}

Assume now that $\mu\in\mathcal E^+$ is bounded. For a given (Borel) set $A\subset X$, denote by $\check{\Lambda}^*_{A,\mu}$ the convex, truncated cone of all $\nu\in\Lambda_{A,\mu}^*$ having property (\ref{leq}).

\begin{theorem}\label{th-ou-fr}Suppose that\/ $\mu\in\mathcal E^+$ is bounded while\/ $\kappa$ satisfies Frostman's maximum principle. Then Definition\/~{\rm\ref{def-bal-ou}} and that obtained from it by replacing\/ $\Lambda^*_{A,\mu}$ by\/ $\check{\Lambda}^*_{A,\mu}$ lead to the same concept of outer balayage.
Furthermore, Theorem\/~{\rm\ref{cor-outer}} and Corollaries\/~{\rm\ref{cor-def11}} and\/ {\rm\ref{cor-def2}} remain valid if\/ $\Lambda^*_{A,\mu}$, $\mathcal E^+_A$, and\/ $\mathcal E'_A$ are replaced throughout by\/ $\check{\Lambda}^*_{A,\mu}$, $\check{\mathcal E}^+_{A,\mu}$, and\/ $\check{\mathcal E}'_{A,\mu}$, respectively.\footnote{For the notations $\check{\mathcal E}^+_{A,\mu}$ and $\check{\mathcal E}'_{A,\mu}$, see the paragraph followed by Theorem~\ref{th-fr}. Also note that Theorem~\ref{th-ou-fr} still holds if (\ref{leq}) is weakened to $\nu(X)\leqslant q\mu(X)$, where $q\in[1,\infty)$ (cf.\ Remark~\ref{weakened}).}\end{theorem}

These results, as well as those in Sect.~\ref{sec-ou-further},
can easily be derived from the theory of inner balayage, developed in Sects.~\ref{sec-inner}--\ref{proofs2} above. For the sake of completeness, we shall nevertheless sketch their proofs.

\subsection{Proofs of Theorems~\ref{cor-outer} and \ref{th-ou-fr} and Corollary~\ref{cor-def2}} The analysis given below is based substantially on the following theorem on capacitability, obtained by a direct application of \cite[Theorem~4.5]{F1}.

\begin{theorem}\label{l-top}Any Borel subset of a perfectly normal, l.c.\ space\/ $X$ of class\/ $K_\sigma$, endowed with a perfect kernel\/ $\kappa$, is capacitable.
\end{theorem}

This enables us to show that for any $\mu\in\mathcal E^+$ and any Borel $A\subset X$, the classes $\Lambda_{A,\mu}$ and $\Lambda^*_{A,\mu}$, appearing in Definitions~\ref{def-bal} and \ref{def-bal-ou} of inner and outer balayage, respectively, coincide:
\begin{equation}\label{theta}\Lambda_{A,\mu}=\Lambda^*_{A,\mu}.\end{equation}
To this end, it is enough to verify that for any $\nu\in\mathcal E^+$ with $\kappa\nu\geqslant\kappa\mu$ n.e.\ on $A$, the same inequality holds q.e.\ on $A$. Being the potential of a (signed) measure of finite energy, $\kappa(\nu-\mu)$ is well defined and finite q.e.\ on $X$ \cite[Corollary to Lemma~3.2.3]{F1}, and it is Borel measurable. Applying Theorem~\ref{l-top} to $E:=A\cap\{x\in X:\ \kappa\nu(x)<\kappa\mu(x)\}$,
we therefore get $c^*(E)=c_*(E)=0$, and (\ref{theta}) follows.

It follows directly from (\ref{theta}) that the inner balayage $\mu^A_*$ (whose existence was justified by Theorem~\ref{th-bal-ex}) gives actually a (unique) solution to the problem on the existence of outer balayage of $\mu$ to $A$:
\begin{equation}\label{11}\mu_*^A=\mu^{*A}.\end{equation}
When substituted into (\ref{proj''}), (\ref{eq-bala-f0}), and (\ref{eq-bala-f2}), this results in (\ref{proj-outer}), (\ref{eq-outer1}), and (\ref{eq-outer2}), respectively.

Applying now Theorem~\ref{l-top} to the (Borel) set $A\cap\{x\in X:\ \kappa\mu_*^A(x)<\kappa\mu(x)\}$, we infer from (\ref{11}) and (\ref{eq-bala-f1}) that
\begin{equation*}\label{111}\kappa\mu^{*A}=\kappa\mu_*^A=\kappa\mu\text{ \ n.e.\ (hence q.e.) on $A$},\end{equation*}
which proves (\ref{eq-outer}). Moreover, $\mu^{*A}$ is the only measure in $\mathcal E_A'$ satisfying (\ref{eq-outer}), for (\ref{eq-bala-f1}) characterizes $\mu_*^A$ uniquely within $\mathcal E_A'$ (see Theorem~\ref{th-bal-ex}).

Having thus verified Theorem~\ref{cor-outer}, assume now that the (Borel) set $A$ is quasiclosed. According to Lemma~\ref{l-quasi}, the convex cone $\mathcal E^+_A$ then coincides with its strong closure $\mathcal E_A'$; hence, the orthogonal projection $\mu_{\mathcal E_A^+}$ exists (Theorem~\ref{th-proj}), and moreover $\mu_{\mathcal E_A^+}=\mu_{\mathcal E_A'}$. Substituting this into Theorem~\ref{cor-outer} yields Corollary~\ref{cor-def2}.

Returning again to arbitrary Borel $A$, suppose finally that $\mu$ is bounded while the kernel satisfies Frostman's maximum principle. Then the inner balayage does not increase the total mass of a measure (Proposition~\ref{cor-mass}), and we have thus been led to (\ref{in1}) and (\ref{in2}). Combining these two with (\ref{theta}) and (\ref{11}) shows that the outer balayage $\mu^{*A}$ belongs, in fact, to both $\check{\mathcal E}_{A,\mu}'$ and $\check{\Lambda}^*_{A,\mu}$, which establishes Theorem~\ref{th-ou-fr} in the same manner as it did in  Sect.~\ref{sec-pr-fr}.

\subsection{Further properties of outer balayage. Convergence assertions}\label{sec-ou-further} Recall that we require the space $X$ to be perfectly normal and of class $K_\sigma$.

\begin{theorem}\label{th-outer-cont} For any\/ $\mu\in\mathcal E^+$, the following assertions\/
{\rm(a)}--{\rm(c)} on  convergence of outer swept measures and their potentials hold true.
\begin{itemize}\item[\rm(a)]For\/ $A$ Borel, the following two limit relations hold when\/ $K\uparrow A$:
\begin{align*}&\mu^{*K}\to\mu^{*A}\text{ \ strongly and vaguely in $\mathcal E^+$},\\
&\kappa\mu^{*K}\uparrow\kappa\mu^{*A}\text{ \ pointwise on\/ $X$}.\end{align*}
\item[\rm(b)]If\/ $A$ is the union of an increasing sequence\/ $(A_k)$ of Borel sets, then
\begin{align*}&\mu^{*A_k}\to\mu^{*A}\text{ \ strongly and vaguely in $\mathcal E^+$},\\
&\kappa\mu^{*A_k}\uparrow\kappa\mu^{*A}\text{ \ pointwise on\/ $X$}.
\end{align*}
\item[\rm(c)]If\/ $A$ is the intersection of a lower directed family\/ {\rm(}resp.\ a decreasing sequence\/{\rm)} $(A_t)$ of closed\/ {\rm(}resp.\ quasiclosed and Borel\/{\rm)} sets, then
\begin{align*}&\mu^{*A_t}\to\mu^{*A}\text{ \ strongly and vaguely in $\mathcal E^+$},\\
&\kappa\mu^{*A_t}\downarrow\kappa\mu^{*A}\text{ \ pointwise q.e.\ on\/ $X$}.
\end{align*}
\end{itemize}
\end{theorem}

\begin{proof}This follows by substituting (\ref{proj-outer}) into Theorems~\ref{th-bal-cont}--\ref{th-cont-bor2}.\end{proof}

\begin{theorem}\label{th-outer-further}Given Borel\/ $A,Q\subset X$ and\/ $\mu\in\mathcal E^+$, the following\/ {\rm(d)--(g)} hold.
\begin{itemize}\item[\rm(d)] {\rm (Monotonicity property)} If\/ $A\subset Q$, then\/
\[\kappa\mu^{*A}\leqslant\kappa\mu^{*Q}.\]
\item[\rm(e)] {\rm (Balayage with a rest)} If\/ $A\subset Q$, then\/
\[\mu^{*A}=(\mu^{*Q})^{*A}.\]
\item[\rm(f)] If\/ $\mu\in\mathcal E'_A$ {\rm(}thus in particular if\/ $\mu\in\mathcal E^+_A${\rm)}, then\/
\[\mu^{*A}=\mu.\]
\item[\rm(g)] Assume Frostman's maximum principle holds. Then\/
\begin{equation}\label{eq-mass20}\mu^{*A}(X)\leqslant\mu(X).\end{equation}
If moreover\/ $c^*(A)<\infty$, then actually
\begin{equation}\label{eq-mass2}\mu^{*A}(X)=\int\kappa\gamma^*_A\,d\mu,\end{equation}
$\gamma^*_A$ being the outer equilibrium measure for\/ $A$.
\end{itemize}
\end{theorem}

\begin{proof}After applying (\ref{proj-outer}) to either of $A$ and $Q$, we deduce (d)--(f) from Propositions~\ref{cor-mon}--\ref{fixx}. If Frostman's maximum principle holds, then combining (\ref{proj-outer}) with Proposition~\ref{cor-mass} leads to (\ref{eq-mass20}). Assume moreover that $c^*(A)<\infty$. The set $A$ being capacitable (Theorem~\ref{l-top}), we conclude from \cite{F1} (Theorems~4.1, 4.3 and Lemma~4.3.4) that the outer and inner equilibrium measures for $A$ (exist and) coincide:
\[\gamma^*_A=\gamma_A.\] Substituting this equality
and (\ref{proj-outer}) into (\ref{eq-mass1}) results in (\ref{eq-mass2}).\end{proof}

\begin{remark}Both (d) and (e) still hold if $Q$ quasicontains $A$, cf.\ Remark~\ref{Q}.\end{remark}

\section{Comments}
\begin{itemize}\item[1.]
The concepts of inner and outer balayage of $\mu\in\mathcal E^+$ to $A\subset X$, introduced by Definitions~\ref{def-bal} and \ref{def-bal-ou}, respectively, and further clarified by a number of subsequent assertions, are in agreement with Cartan's concepts of inner and outer Newtonian balayage on $\mathbb R^n$, $n\geqslant3$ (cf.\ \cite[Section~19, Theorems~1, 1$'$]{Ca2}).
\item[2.] If the space $X$ has a countable base of open sets, then another approach to outer balayage of $\mu\in\mathcal E^+$ to arbitrary $A\subset X$ was suggested by Fuglede \cite[Theorem~4.15]{Fu5}. The outer balayage to a set $A$ was defined there as that to its quasiclosure (for the concept of quasiclosure, see \cite[Section~2.8]{F71}).
     Our approach to outer balayage is relevant to a wider class of l.c.\ spaces $X$,\footnote{Indeed, according to \cite[Section~IX.2, Corollary to Proposition~16]{B3}, a l.c.\ space $X$ has a countable base of open sets if and only if it is metrizable and of class $K_\sigma$. Being therefore metrizable, a sec\-ond-count\-able, l.c.\ space $X$ must be perfectly normal \cite[Section~IX.1, Proposition~2]{B3}, whereas the converse is false in general \cite[Section~IX.2, Exercise~13(b)]{B3}.} though being limited only to Borel sets $A$. But when these two approaches can be applied simultaneously, they turned out to be equivalent. (A concept of inner balayage, basic to the
present study, was not considered in \cite{Fu5}.)\end{itemize}

\section{Acknowledgements} The author thanks Professor Edward B.~Saff for a fruitful discussion on relevant questions in the theory of logarithmic potentials on the plane.

\end{document}